\documentclass[11pt]{article}
\usepackage[margin=1in]{geometry} 
\usepackage{amssymb,amsfonts,amsmath,bbm,mathrsfs,stmaryrd}
\usepackage{url}

\usepackage[colorlinks,
             linkcolor=black!75!red,
             citecolor=blue,
             pdftitle={Three results on representations of Mackey Lie algebras},
             pdfauthor={Alexandru Chirvasitu},
             pdfsubject={rings and algebras},
             pdfkeywords={rt.representation theory, ct.category theory},
             pdfproducer={pdfLaTeX},
             pdfpagemode=None,
             bookmarksopen=true
             bookmarksnumbered=true]{hyperref}

\usepackage{tikz}
\usetikzlibrary{arrows,decorations.pathreplacing,decorations.markings,shapes.geometric,through,fit,shapes.symbols,positioning}

\usepackage[amsmath,thmmarks,hyperref]{ntheorem}
\usepackage{cleveref}

\crefname{section}{Section}{Sections}
\crefformat{section}{#2Section~#1#3} 
\Crefformat{section}{#2Section~#1#3} 

\crefname{subsection}{\S}{\S\S}
\crefformat{subsection}{#2\S#1#3} 
\Crefformat{subsection}{#2\S#1#3} 

\theoremstyle{plain}

\newtheorem{lemma}{Lemma}[section]
\newtheorem{proposition}[lemma]{Proposition}
\newtheorem{corollary}[lemma]{Corollary}
\newtheorem{theorem}[lemma]{Theorem}
\newtheorem{conjecture}[lemma]{Conjecture}

\theoremstyle{nonumberplain}

\theoremstyle{plain}
\theorembodyfont{\upshape}
\theoremsymbol{\ensuremath{\blacklozenge}}

\newtheorem{definition}[lemma]{Definition}

\newtheorem{remark}[lemma]{Remark}

\crefname{definition}{definition}{definitions}
\crefformat{definition}{#2definition~#1#3} 
\Crefformat{definition}{#2Definition~#1#3} 

\crefname{lemma}{lemma}{lemmas}
\crefformat{lemma}{#2lemma~#1#3} 
\Crefformat{lemma}{#2Lemma~#1#3} 

\crefname{proposition}{proposition}{propositions}
\crefformat{proposition}{#2proposition~#1#3} 
\Crefformat{proposition}{#2Proposition~#1#3} 

\crefname{ex}{example}{examples}
\crefformat{example}{#2example~#1#3} 
\Crefformat{example}{#2Example~#1#3} 

\crefname{remark}{remark}{remarks}
\crefformat{remark}{#2remark~#1#3} 
\Crefformat{remark}{#2Remark~#1#3} 

\crefname{corollary}{corollary}{corollaries}
\crefformat{corollary}{#2corollary~#1#3} 
\Crefformat{corollary}{#2Corollary~#1#3} 

\crefname{theorem}{theorem}{theorems}
\crefformat{theorem}{#2theorem~#1#3} 
\Crefformat{theorem}{#2Theorem~#1#3} 

\crefname{equation}{}{}
\crefformat{equation}{(#2#1#3)} 
\Crefformat{equation}{(#2#1#3)}

\theoremstyle{nonumberplain}
\theoremsymbol{\ensuremath{\blacksquare}}

\newtheorem{proof}{Proof}
\newtheorem{proof of main}{Proof of \Cref{th.main}}
\newtheorem{proof of res1}{Proof of \Cref{th.res1}}
\newtheorem{proof of res2}{Proof of \Cref{th.res2}}

\newcommand\bC{{\mathbb C}}
\newcommand\bF{{\mathbb F}}
\newcommand\bN{{\mathbb N}}

\newcommand\bT{{\mathbb T}}
\newcommand\bZ{{\mathbb Z}}

\newcommand\fg{\mathfrak{g}}
\newcommand\fG{\mathfrak{G}}
\newcommand\fh{\mathfrak{h}}

\newcommand\fgl{\mathfrak{gl}}
\newcommand\fsl{\mathfrak{sl}}

\newcommand\cC{{\mathcal C}}

\DeclareMathOperator{\End}{\mathrm{End}}
\DeclareMathOperator{\wt}{\mathrm{wt}}

\newcommand{\define}[1]{{\em #1}}

\newcommand{\cat}[1]{\textsc{#1}}

\newcommand{\qedhere}{\mbox{}\hfill\ensuremath{\blacksquare}}

\title{Three results on representations of Mackey Lie algebras}
\author{Alexandru Chirvasitu}

\begin{document}

\date{}

\maketitle

\begin{abstract}
I. Penkov and V. Serganova have recently introduced, for any non-degenerate pairing $W\otimes V\to\bC$ of vector spaces, the Lie algebra $\fgl^M=\fgl^M(V,W)$ consisting of endomorphisms of $V$ whose duals preserve $W\subseteq V^*$. In their work, the category $\bT_{\fgl^M}$ of $\fgl^M$-modules which are finite length subquotients of the tensor algebra $T(W\otimes V)$ is singled out and studied. In this note we solve three problems posed by these authors concerning the categories $\bT_{\fgl^M}$. Denoting by $\bT_{V\otimes W}$ the category with the same objects as $\bT_{\fgl^M}$ but regarded as $V\otimes W$-modules, we first show that when $W$ and $V$ are paired by dual bases, the functor $\bT_{\fgl^M}\to \bT_{V\otimes W}$ taking a module to its largest weight submodule with respect to a sufficiently nice Cartan subalgebra of $V\otimes W$ is a tensor equivalence. Secondly, we prove that when $W$ and $V$ are countable-dimensional, the objects of $\bT_{\End(V)}$ have finite length as $\fgl^M$-modules. Finally, under the same hypotheses, we compute the socle filtration of a simple object in $\bT_{\End(V)}$ as a $\fgl^M$-module. 
\end{abstract}


\tableofcontents


\section*{Introduction}

In the recent paper \cite{Mackey}, the authors study various categories of representations for Lie algebras associated to pairs of complex vector spaces $V,W$ endowed with a non-degenerate bilinear form $W\otimes V\to \bC$. This datum realizes $W$ as a subspace of the full dual $V^*$, and vice versa. The associated Mackey Lie algebra $\fg^M=\fgl^M(V,W)$ is then simply the set of all endomorphisms of $V$ whose duals leave $W\subseteq V^*$ invariant. It can be shown that the definition is symmetric in the sense that reversing the roles of $V$ and $W$ produces canonically isomorphic Lie algebras. When $W=V^*$, the resulting Lie algebra is simply $\End(V)$.  

Categories $\bT_{\fg^M}$ of $\fg^M$-representations are then introduced. They consist of modules for which all elements have appropriately large annihilators; see \cite[$\S$ 7.3]{Mackey}. One remarkable result is that all these categories, for all possible non-degenerate pairs $(V,W)$, are in fact equivalent as tensor categories (i.e. symmetric monoidal abelian categories). Moreover, they are also equivalent to the categories $\bT_{\fsl(V,W)}$ from \cite[$\S$ 3.5]{Mackey} and $\bT=\bT_{\fsl(\infty)}$ introduced and studied earlier in \cite{DPS}; all of this follows from \cite[Theorems 5.1 and 7.9]{Mackey}. 

In view of the abstract equivalence between $\bT_{\End(V)}\simeq \bT_{\fsl(V,W)}$ noted above, it is a natural problem to try to find as explicit and natural a functor as possible that implements this equivalence. In order to do this, we henceforth specialize to the case when $W=V_*$ is a vector space whose pairing with $V$ is given by a pair of dual bases $v_\gamma\in V$, $v_\gamma^*\in V_*$ for $\gamma$ ranging over some (possibly uncountable) set $I$. This assumption ensures the existence of a so-called \define{local Cartan subalgebra} $\fh\subseteq\fsl(V,V_*)$ \cite[1.4]{Mackey}.  

Denote $\fg=\fsl(V,V_*)$. In our setting, for a local Cartan subalgebra $\fh\subseteq\fg$, let $\Gamma^{\wt}_{\fh}$ be the functor from $\End(V)$-modules to $\fg$-modules which picks out the $\fh$-weight part of a representation. Similarly, denote by $\Gamma^{\wt}$ the functor $\bigcap_{\fh}\Gamma^{\wt}_{\fh}$, where the intersection ranges over all local Cartan subalgebras of $\fg$. We will abuse notation and denote by these same symbols the restrictions of $\Gamma^{\wt}_{\fh}$ and $\Gamma^{\wt}$ to various categories of ${\End(V)}$ - modules.  

With these preparations (and keeping the notations we've been using), the following seems reasonable (\cite[8.4]{Mackey}).

\begin{conjecture}\label{conj.Mackey}
The functor $\Gamma^{\wt}$ implements an equivalence from $\bT_{\End(V)}$ onto $\bT_{\fg}$.  
\end{conjecture}

One of the main results of this note is a proof of this conjecture. The outline of the note is as follows: 

In the next section we prove \Cref{conj.Mackey}, making use of the results in \cite{SS} on a certain universality property for the category $\bT_{\fg}$.

In \Cref{se.res1} we specialize to a pairing $V_*\otimes V\to\bC$ of countable-dimensional vector spaces $V$, $V_*$. In this case, noting that $V^*/V_*$ is a simple $\fg^M=\fsl^M(V,V_*)$-module, the authors of \cite{Mackey} ask whether {\it all} objects of $\bT_{\End(V)}$ are finite-length $\fg^M$-modules. We show that this is indeed the case in \Cref{th.res1}.

Finally, \Cref{th.res2} in \Cref{se.res2} contains the description of the socle filtration as a $\fg^M$-module of a simple object in $\bT_{\End(V)}$. This solves a third problem posed in the cited paper.

\subsection*{Acknowledgements}

I would like to thank Ivan Penkov and Vera Serganova for useful discussions on the contents of \cite{DPS,Mackey} and for help editing the manuscript. 

This work was partially supported by the Danish National Research Foundation through the QGM Center at Aarhus University, and by the Chern-Simons Chair in Mathematical Physics at UC Berkeley.

\section{Explicit equivalence between \texorpdfstring{$\bT_{\End(V)}$}{TEnd(V)} and \texorpdfstring{$\bT_{\fg}$}{Tg}}\label{se.main}

We will actually prove a slightly strengthened version of \Cref{conj.Mackey}. Before formulating it, recall our setting: We are considering a pairing between $V$ and $V_*$ determined by dual bases $v_\gamma\in V$ and $v_\gamma^*\in V_*$, and $\fg$ stands for $\fsl(V,V_*)$. By $\fh\subseteq\fg$ we denote the local Cartan subalgebra spanned linearly by the elements $v_\gamma\otimes v_\gamma^*-v_{\bar\gamma}\otimes v_{\bar\gamma}^*\in \fg\subseteq V\otimes V_*$ for $\gamma\ne\bar\gamma\in I$.  

Throughout, `tensor category' means symmetric monoidal, $\bC$-linear and abelian. Similarly, tensor functors are symmetric monoidal and $\bC$-linear, and tensor natural transformations are symmetric monoidal.  

Our main result in the present section reads as follows.

\begin{theorem}\label{th.main}
The functor $\Gamma^{\wt}_{\fh}$ implements a tensor equivalence from $\bT_{\End(V)}$ onto $\bT_{\fg}$. 
\end{theorem}

Before embarking on the proof, note that as claimed above, the theorem implies the conjecture.

\begin{corollary}
\Cref{conj.Mackey} is true. 
\end{corollary}
\begin{proof}
On the one hand, the functor $\Gamma^{\wt}$ from the statement of the conjecture is a subfunctor of $\Gamma^{\wt}_{\fh}$. On the other hand though, the theorem says that $\Gamma^{\wt}_{\fh}$ already lands inside the category $\bT_{\fg}$ which consists of weight modules for any local Cartan subalgebra of $\fg$ (because it consists of modules embeddable in finite direct sums of copies of the tensor algebra $T(V\oplus V_*)$; see e.g. \cite[7.9]{Mackey}). In conclusion, we must have $\Gamma^{\wt}=\Gamma^{\wt}_{\fh}$, and we are done.  
\end{proof}

We will make use of the following simple observation.

\begin{lemma}\label{le.main}
Let $H$ be a cocommutative Hopf algebra over an arbitrary field $\bF$, and $\cat{Fin}:H-\mathrm{mod}\to H-\mathrm{mod}$ the functor sending an $H$-module $M$ to the largest $H$-submodule of $M$ which is a union of finite-dimensional $H$-modules. Then, $\cat{Fin}$ is a tensor functor.  
\end{lemma}
\begin{proof}
Let $S$ be the antipode of $H$. For an $H$-module $V$, denote by $V^*$ the algebraic dual of $V$ made into an $H$-module via $(hf)(v)=f(S(h)v)$ for $h\in H$, $v\in V$ and $f\in V^*$. Then the usual evaluation $V^*\otimes V\to \bF$ is an $H$-module map if $V^*\otimes V$ is a module via the tensor cateory structure of $H-\mathrm{mod}$. 

Now let $M,N$ be $H$-modules, and $V\subseteq \cat{Fin}(M\otimes N)$ be a finite-dimensional $H$-submodule. We need to show that $V$ is in fact a submodule of $\cat{Fin}(M)\otimes\cat{Fin}(N)$. 

Denote by $N_V\subseteq N$ the image of the $H$-module morphism $M^*\otimes V\subseteq M^*\otimes M\otimes N\to N$, where the last arrow is evaluation on the first two tensorands. Similarly, denote by $M_V\subseteq M$ the image of $V\otimes N^*\subseteq M\otimes N\otimes N^*\to M$. Then $M_V$ and $N_V$ are $H$-submodules of $M$ and $N$ respectively, being images of module maps. It is now easily seen from their definition that $M_V$ and $N_V$ are finite-dimensional, and that the inclusion $V\subseteq M\otimes N$ factors through $M_V\otimes N_V\subseteq M\otimes N$.   
\end{proof}

\begin{remark}
The cocommutativity of $H$ is used in the proof to conclude that the category $H-\mathrm{mod}$ is symmetric monoidal, and hence $N\otimes N^*\to \bF$ is an $H$-module map because its domain is isomorphic to $N^*\otimes N$.

Although we do not need this in the sequel, as \Cref{le.main} will only be applied to universal envelopes of Lie algebras, the above proof can be generalized to show that the functor $\cat{Fin}$, defined in the obvious fashion, is monoidal for any Hopf algebra with bijective antipode. In the definition of $V_N$ one would need to use the evaluation map $N\otimes{^*N}\to \bF$ instead, where $^*N$ is the full dual of $N$ made into an $H$-module using the inverse of the antipode instead of the antipode.    
\end{remark}

We now need a characterization of the category $\bT_{\End(V)}\simeq \bT_{\fg}$ in terms of a universality property which defines it uniquely up to tensor equivalence. The following result is Theorem 3.4.2 from \cite{SS}, where the category $\bT_{\fg}$ is denoted by $\mathrm{Rep}({\bf GL})$.

\begin{theorem}\label{th.SS}
For any tensor category $\cC$ with monoidal unit ${\bf 1}$ and any morphism $b:x\otimes y\to{\bf 1}$ in $\cC$, there is a left exact tensor functor $F:\bT_{\End(V)}\to\cC$ sending the pairing $V^*\otimes V\to\bC$ in $\bT_{\End(V)}$ to $b$. Moreover, $F$ is unique up to tensor natural isomorphism. \qedhere
\end{theorem}

As an immediate consequence we have:

\begin{corollary}\label{co.main}
A left exact tensor functor $\bT_{\End(V)}\to\bT_{\fg}$ turning the pairing $V^*\otimes V\to\bC$ into the pairing $V_*\otimes V\to\bC$, is a tensor equivalence.  
\end{corollary}
\begin{proof} 
The abstract tensor equivalence $\bT_{\End(V)}\simeq\bT_{\fg}$ established in \cite[5.1,7.9]{Mackey} identifies the two bilinear pairings in the statement. The conclusion then follows from \Cref{th.SS} in the usual manner (a universality property implies uniqueness up to equivalence). 
\end{proof}

The proof of \Cref{th.main} makes use of the following auxiliary result.

\begin{lemma}\label{le.weight_prod}
Let $\fh$ be a complex abelian Lie algebra. For any functional $\varphi\in\fh^*$, let $M^\varphi\in\fh-\mathrm{mod}$ be an $\fh$-module all of whose elements are vectors of weight $\varphi$. Then, using the notation $\cat{Fin}$ from \Cref{le.main}, we have
\[
	\cat{Fin}\left(\prod_{\varphi\in\fh^*}M^\varphi\right) = \bigoplus_{\varphi\in\fh^*}M^\varphi.
\] 
\end{lemma}
\begin{proof}
We denote the direct product $\displaystyle \prod_{\varphi}M^\varphi$ by $M$. Let $x\in M$ be an element contained in some $d$-dimensional $\fh$-submodule $N$ of $M$. 

Assume there are $d+1$ distinct functionals $\varphi_0$ up to $\varphi_d$ such that the components $x_i$, $0\le i\le d$ of $x$ in $M^{\varphi_i}$ are all non-zero. Because $\varphi_i\in\fh^*$ are distinct, we can find some element $h\in\fh$ such that the scalars $t_i=\varphi_i(h)$ are distinct (as $\fh$ cannot be the union of the kernels of $\varphi_i-\varphi_j$, $0\le i\ne j\le d$; here we use the fact that we are working over $\bC$, or more generally, over an infinite field). The claim now is that $x,hx,\ldots,h^dx$ are linearly independent, contradicting the assumption $\dim N=d$. 

To prove the claim, consider the images of the vectors $h^ix$, $0\le i\le d$ through the projection $\displaystyle M\to\prod_{i=0}^dM^{\varphi_i}$. They are linear combinations of the $x_i$'s, and their coefficients form the columns of the $(d+1)\times(d+1)$ non-singular Vandermonde matrix
\[
	\begin{pmatrix}
	1      & t_0    & \cdots & t_0^d  \\
	1      & t_1    & \cdots & t_1^d  \\
	\vdots & \vdots & \ddots & \vdots \\
	1      & t_d    & \cdots & t_d^d
	\end{pmatrix}
\]
This finishes the proof. 
\end{proof}

We are now ready to prove our first main result.

\begin{proof of main}
First, recall from \cite[$\S$4]{DPS} that $\bT_{\End(V)}$ has enough injectives, that the tensor products $(V^*)^{\otimes m}\otimes V^{\otimes n}$ contain all indecomposable injectives as summands, and also \cite{PS} that all morphisms between such tensor products are built out of the pairing $V^*\otimes V\to\bC$ by taking tensor products, permutations, and linear combinations. Therefore, if we show that $\Gamma=\Gamma^{\wt}_{\fh}:\bT_{\End(V)}\to \fg-\mathrm{mod}$ sends $V^*\otimes V\to\bC$ to $V_*\otimes V\to\bC$ and is a left exact tensor functor, its image will automatically lie in $\bT_{\fg}$. We can then apply \Cref{co.main} to conclude that the resulting functor from $\bT_{\End(V)}\to\bT_{\fg}$ is a tensor equivalence.

The functor $\Gamma$, regarded as a functor from $\fg-\mathrm{mod}$ to $\fh$-weight $\fg$-modules, is the right adjoint of the exact inclusion functor going in the opposite direction; it's thus clear that it is left exact. 

Since the pairing $V_*\otimes V\to\bC$ is simply the restriction of the full pairing $V^*\otimes V\to\bC$ and $\Gamma$ is compatible with inclusions, we will be done as soon as we prove that it is a tensor functor and it sends $V^*$ to $V_*$.

We prove tensoriality first. In fact, since compatibility with the symmetry is clear, it is enough to prove monoidality. That is, that the inclusion $\Gamma(M)\otimes\Gamma(N)\subseteq\Gamma(M\otimes N)$ is actually an isomorphism for any $M,N\in\bT_{\End(V)}$. To see that this is indeed the case, note that every object of $\bT_{\End(V)}$, being embedded in some finite direct sum of tensor products $(V^*)^{\otimes m}\otimes V^{\otimes n}$, is certainly a submodule of a direct product of $\fh$-weight spaces. \Cref{le.weight_prod} now shows that every finite-dimensional $\fh$-submodule of an object in $\bT_{\End(V)}$ is automatically an $\fh$-weight module. Conversely, $\fh$-weight modules are unions of finite-dimensional $\fh$-modules. It follows that $\Gamma$ coincides with the functor $\cat{Fin}$ considered in \Cref{le.main} for the Hopf algebra $H=U(\fh)$ (i.e. the universal enveloping algebra of $\fh$); the lemma finishes the job of proving monoidality. 

Finally, it is almost immediate that $\Gamma(V^*)=V_*$: simply note that the $\fh$-weight subspaces of $V^*$ are the lines spanned by the basis elements $v_\gamma^*$.
\end{proof of main}

\section{Restrictions from \texorpdfstring{$\bT_{\End(V)}$}{TEnd(V)} to \texorpdfstring{$\fg^M(V,V_*)$}{gM(V,V*)} have finite length}\label{se.res1}

In what follows $V_*\otimes V\to\bC$ will be a non-degenerate pairing between countable-dimensional vector spaces. In this case, it is shown in \cite{M} that we can find dual bases $v_i$, $v_i^*$, $i\in\bN=\{0,1,\ldots\}$ for $V$ and $V_*$ respectively, in the sense that $v_i^*(v_j)=\delta_{ij}$.  

Denote $\fg=\fsl(V,V_*)$ and $\fg^M=\fsl^M(V,V_*)$. In general, for a vector space $W$ and a partition $\lambda=(\lambda_1\ge \lambda_2\ge\ldots\ge\lambda_m\ge 0)$, denote by $W_\lambda$ the image of the Schur functor corresponding to $\lambda$ applied to $W$. 

We think of elements of $V^*$ as row vectors indexed by $\bN$, on which the Lie algebra $\End(V)$ of $\bN\times\bN$ matrices with finite columns acts on the left via $-$ right multilication. The subspace $V_*\subseteq V^*$ consists of row vectors with only finitely many non-zero entries. We will often think of elements of $V^*/V_*$ as row vectors as well, keeping in mind that changing finitely many entries does not alter the element. For a subset $I\subseteq\bN$ and a vector $x\in V^*$, the \define{restriction} $x|_I$ is the vector obtained by keeping the entries of $x$ indexed by $I$ intact, and turning all other entries to zero. The same terminology applies to $x\in V^*/V_*$. 

Let $I\subseteq\bN$ be a subset. An element of $V^*$ (respectively $V^*/V_*$) is \define{I-concentrated}, or \define{concentrated in $I$} if all of its non-zero entries (respectively all but finitely many of its non-zero entries) belong to $I$. Similarly, a matrix in $\End(V)$ is $I$-concentrated if all of its non-zero entries are in $I\times I$.  

Our result is:

\begin{theorem}\label{th.res1}
The objects of $\bT_{\End(V)}$ have finite length when regarded as $\fg^M$-modules. 
\end{theorem}

\begin{remark}\label{re.uncount}
In fact, the proof of the theorem could be adapted to the more general case covered in the previous section: $V_*$ and $V$ could be allowed to be uncountable-dimensional, so long as they are still paired by means of dual bases.  
\end{remark}

We need some preparations. The following result is very likely well known.

\begin{lemma}\label{le.tens_prod}
Let $\fG$ be a Lie algebra over some field $\bF$, and $I\subseteq\fG$ an ideal. Let $U$ be a simple $\fG/I$-module, and $W$ an $\fG$-module on which $I$ acts densely and irreducibly, and such that $\End_I(W)=\bF$. Then, $U\otimes W$ is a simple $\fG$-module. 
\end{lemma}
\begin{proof}
Let $x=\sum_{i=1}^n u_i\otimes w_i$ be a non-zero element of $U\otimes W$, with the tensor product decomposition chosen such that the $u_i$ are linearly independent and all $w_i$ are non-zero. We have to show that $x$ generates $U\otimes W$ as an $\fG$-module.

Because $I$ annihilates $U$, it acts on $\bigoplus_{i=1}^n \bF u_i\otimes W\cong W^{\oplus n}$. By the simplicity of $W$ over $I$, there are vectors $w'_i\in W$, $i\in\{1,\ldots,n\}$, with $w'_1=w_1$ and such that the projection $W^{\oplus n}\to W$ onto the first component maps the $I$-submodule of $W^{\oplus n}$ generated by $\sum u_i\otimes w'_i$ isomorphically onto $W$. 

Now note that $w'_1\mapsto w'_i$, $i>1$ extend to $I$-module automorphisms of $W$. By the condition $\End_I(W)=\bF$, these automorphisms are scalar: $w'_i=t_i w'_1=t_i w_1$ for all $i>1$; for simplicity, set $t_1=1$ so that this identity holds for all $i$. Now, substituting $\sum_i t_i u_i$ for $u_1$ and denoting $u_1=u$, we may assume that a non-zero simple tensor $u\otimes w\in U\otimes W$ belongs to the $I$-span of $x$. 

Starting with a simple tensor $u\otimes w$ as above, note first that the enveloping algebra $U(I)$ can act so as to obtain any other tensor of the form $u\otimes w'$ (because $I$ annihilates $U$ and acts irreducibly on $W$). 

On the other hand, for $h\in\fG$, we have $h(u\otimes w)=hu\otimes w+u\otimes hw$. Since $I$ acts densely on $W$, we can find $k\in I$ such that $kw=hw$. In this case we have $(h-k)(u\otimes w)=hu\otimes w$; since $U$ is simple over $\fG$, all simple tensors of the form $u'\otimes w$ are in the $\fG$-span of $u\otimes w$. Combining this with the previous paragraph, we get the desired conclusion.  
\end{proof}

We now need the following infinite-dimensional analogue of Schur-Weyl duality.

\begin{proposition}\label{pr.Schur}
For any partition $\lambda=(\lambda_1\ge\ldots\ge\lambda_m\ge 0)$, the $\fg^M$-module $(V^*/V_*)_\lambda$ is simple.
\end{proposition}
\begin{proof}
Let $\displaystyle k=|\lambda|=\sum_{i=1}^m\lambda_i$. Choose an arbitrary non-zero $x\in (V^*/V_*)^{\otimes k}$, thought of as a sum $\sum_\ell x^\ell$ for $x^\ell=x^\ell_1\otimes\ldots\otimes x^\ell_k$. 

We denote the symmetric group on $k$ letters by $S_k$. Partition $\bN$ into $k$ infinite subsets $I_1,\ldots, I_k$ such that the element of $(V^*/V_*)^{\otimes k}$ defined by 
\[
	x_{\cat{res}}=\sum_\ell \sum_{\sigma\in S_k} (x^\ell_1|_{I_{\sigma(1)}})\otimes\ldots\otimes (x^\ell_k|_{I_{\sigma(k)}}) 
\] 
is non-zero; we leave it to the reader to show that this is possible. Now choose $k$ complex numbers $t_j$, $j\in\{1,\ldots,k\}$ such the sums $\sum_j m_j t_j$ for non-negative integers $\sum_j m_j=k$ are distinct for different choices of tuples $m_1,\ldots, m_k$ (e.g. $t_j$ could equal $(k+1)^j$), and let $h\in\fg^M$ be the diagonal matrix whose $I_j$-indexed entries are equal to $t_j$. By breaking everything up into $h$-eigenspaces, we see that the $\fg^M$-module generated by $x$ contains $x_{\cat{res}}$. In order to keep notation simple, we substitute $x_{\cat{res}}$ for $x$ and assume that the individual tensorands $x^\ell_j$ of each summand $x^\ell$ of $x$ are concentrated in distinct $I_j$'s. 

Now consider the subspace $W_1$ of $V^*/V_*$ generated by all $I_1$-concentrated $x^\ell_j$'s, and let $p_1,\ldots, p_s$ be rank one idempotent $I_1$-concentrated matrices in $\fg^M$ such that $\sum_i p_i$ acts as the identity on $W_1$. Since $x=\sum_i p_i x$, some $p_ix$ must be non-zero. Substitute it for $x$, and repeat the process with $I_2$ in place of $I_1$, etc. The resulting non-zero element, again denoted by $x$, will now be a linear combination of simple tensors $x^\ell$ as before, with tensorands $x^\ell_j$ concentrated in distinct $I_j$'s for each $\ell$, and such that all $x^\ell_i$'s concentrated in $I_j$ (for all $\ell$) are equal. Denoting by $x_j$ this common $I_j$-concentrated vector, our element $x$ is a linear combination of permutations of $x_1\otimes\ldots\otimes x_k$. 

Note that the entire procedure we have just described is $S_k$-equivariant: If the vector we started out with was in $(V^*/V_*)_\lambda\subseteq (V^*/V_*)^{\otimes k}$, then so is the output of the process. We assume this to be the case for the rest of the proof.

Because $\fg^M$ acts transitively on $V^*/V_*$ (in the sense that any non-zero element can be transformed into any other element by acting on it with some matrix in $\fg^M$), we can find $k^2$ elements $a_{ij}$ of $\fg^M$ such that $a_{pq}x_r=\delta_{qr}x_p$. The elements $a_{ij}$ generate a Lie algebra isomorphic to $\fgl(k)$, and by ordinary Schur-Weyl duality we conclude that the $\fg^M$-module generated by $x$ contains $c_\lambda(W^{\otimes k})$, where $W$ is the linear space spanned by the $x_j$, and $c_\lambda$ is the Young symmetrizer corresponding to $\lambda$.   

Since for each $j$ the vector $x_j$ can be transformed into any other $I_j$-concentrated vector by acting on it with some $I_j$-concentrated matrix, the conclusion from the previous paragraph applies to any choice of $x_j$'s. The desired result follows from the fact that every element of $c_\lambda(V^*/V_*)^{\otimes k}$ is a sum of elements from $c_\lambda(W^{\otimes k})$ for various $W$ spanned by various tuples $\{x_j\}$.  
\end{proof}

As a consequence of \Cref{le.tens_prod} and \Cref{pr.Schur} we get:

\begin{corollary}\label{co.res1}
Let $W$ be a simple object in $\bT_{\fg^M}$, and $\lambda$ be a partition. Then, the $\fg^M$-module $(V^*/V_*)_\lambda\otimes W$ is simple.
\end{corollary}
\begin{proof}
We apply \Cref{le.tens_prod} to the Lie algebra $\fG=\fg^M$, the ideal $I=\fg$, and the modules $U=(V^*/V_*)_\lambda$ and $W$. We already know that $W$ is simple over $\fg$ and is acted upon densely by the latter Lie algebra (\cite[Corollary 7.6]{Mackey}), and the remaining condition $\End_{\fg}(W)=\bC$ follows for example from the fact that all simple modules in $\bT_{\fg}=\bT_{\fg^M}$ are highest weight modules with respect to a certain Borel subalgebra of $\fg$.  
\end{proof}

We can now turn to \Cref{th.res1}.

\begin{proof of res1}
Since all objects of $\bT_{\End(V)}$ are isomorphic to subquotients of finite direct sums of tensor products $(V^*)^{\otimes m}\otimes V^{\otimes n}$, it suffices to prove the conclusion for these tensor products. In turn, when regarded as $\fg^M$-modules, these tensor products have filtrations by finite direct sums of objects of the form $(V^*/V_*)^{\otimes m_1}\otimes V_*^{\otimes m_2}\otimes V^{\otimes n}$. The leftmost tensorand $(V^*/V_*)^{\otimes m_1}$ breaks up as a direct sum of images $(V^*/V_*)_\lambda$ of Schur functors, while $V_*^{\otimes m_2}\otimes V^{\otimes n}$ has a finite filtration by simple modules from the category $\bT_{\fg^M}$. The conclusion now follows from \Cref{co.res1} above.   
\end{proof of res1}

\section{Socle filtrations of \texorpdfstring{$\bT_{\End(V)}$}{TEnd(V)}-objects over \texorpdfstring{$\fg^M(V,V_*)$}{gM(V,V*)}}\label{se.res2}

We now tackle the problem of finding the socle filtrations of simples in $\bT_{\End(V)}$ as $\fg^M$-modules. We start with a definition.

\begin{definition}\label{de.defess}
A filtration  $M^0\subseteq M^1\subseteq\ldots\subseteq M^n=M$ of an object $M$ in an abelian category is \define{essential} if for every $p<q<r$, the module $M^q/M^p$ is essential in $M^r/M^q$, i.e. intersects every non-zero submodule of $M^r/M^q$ non-trivially.
\end{definition}

\begin{remark}\label{re.defess}
It can be shown by induction on $r-p$ that the condition in \Cref{de.defess} is equivalent to $M^{p+1}/M^p$ being essential in $M^{p+2}/M^p$ for all $p$. 
\end{remark}

For dealing with tensor products of copies of $V^*$ and $V_*$ we use the following notation: For a binary word ${\bf r}=(r_1,\ldots,r_k)$, $r_i\in\{0,1\}$, let $V^{\bf r}$ be the tensor product $\bigotimes_{i=1}^k V^{r_i}$, where $V^0=V_*$ and $V^1=V^*$. We denote $\sum_i r_i$ by $|{\bf r}|$. 

Now consider the following (ascending) filtration of $W=(V^*)^{\otimes m}\otimes V^{\otimes n}$:
\begin{equation}\label{eq.filt}
	W^k=\sum_{|{\bf r}|\subseteq k}V^{\bf r}\otimes V^{\otimes n},\ \text{ for every } 0\subseteq k\subseteq m. 
\end{equation}

\begin{proposition}\label{pr.ess_filt}
The filtration \Cref{eq.filt} of $W=(V^*)^{\otimes m}\otimes V^{\otimes n}$ is essential in ${\fg^M}-\mathrm{mod}$. 
\end{proposition}
\begin{proof}
By \Cref{re.defess}, it suffices to show that for any $k\ge 0$, the $\fg^M$-module generated by any element  $x\in W^{k+2}-W^{k+1}$ intersects $W^{k+1}-W^k$ (the minus signs stand for set difference). Moreover, it is enough to assume that $x$ is a sum of simple tensors $y=y_1\otimes\ldots\otimes y_m\otimes x_1\otimes\ldots\otimes x_n$ for $y_i\in V^*$ and $x_j\in V$ such that exactly $k+2$ of the $y_i$'s are in $V_*$. 

Acting on a term $y$ as above with an element $g$ of $\fg$ which annihilates all $y_i\in V_*$ and all $x_j$ will produce an element of $W^{k+1}$, which belongs to $W^{k+1}-W^k$ provided it is non-zero; this element can be written as a sum of simple tensors, each of which has the tensorands $y_i\in V_*$ and $x_j$ in common with the original term $y$. Hence, focusing on the action of $g$ on only those tensorands $y_i$ which do not belong to $V_*$, it is enough to prove the following claim (which we apply to $s=m-(k+2)$): 

The annihilator in $\fg$ of an element $z\in (V^*)^{\otimes s}$ whose image in $(V^*/V_*)^{\otimes s}$ is non-zero does not contain a finite corank subalgebra (as defined in \cite[$\S$ 3.5]{Mackey}). 

Fixing $p\in\bN$, we have to prove that some matrix $a\in \fg$ concentrated in $\bN_{\ge p}=\{p,p+1,\ldots\}$ does not annihilate $z$. In fact, it is enough to prove this for $a\in\fg^M$. Indeed, it would then follow that for sufficiently large $q>p$, the vector $(az)_{\le q}$ obtained by annihilating all coordinates with index larger than $q$ is non-zero. But we can find some large $r$ such that $(az)_{\le q}$ equals $(a_{\le r}z)_{\le r}$, where $a_{\le r}$ is the $\{p,\ldots,r\}$-concentrated truncation of $a$. We would then conclude that $a_{\le r}z$ is non-zero, and the proof would be complete. 

Finally, to show that some $\bN_{\ge p}$-concentrated $a\in\fg^M$ does not annihilate $z$, it suffices to pass to the quotient by $V_*$, and regard $z$ as a non-zero element of $(V^*/V_*)^{\otimes s}$. Since $\fg^M$ acts on $V^*/V_*$ via its quotient $\fg^M/\fg$, being $\bN_{\ge p}$-concentrated no longer matters: any element of $\fg^M$ can be brought into $\bN_{\ge p}$-concentrated form by adding an element of $\fg$. In conclusion, the desired result is now simply that no non-zero element of $(V^*/V_*)^{\otimes s}$ is annihilated by $\fg^M$; this follows immediately from \Cref{pr.Schur}, for example.   
\end{proof}

The proof is easily applicable to traceless tensors in $(V^*)^{\otimes m}\otimes V^{\otimes n}$, i.e. the intersection of the kernels of all $mn$ evaluation maps \[ (V^*)^{\otimes m}\otimes V^{\otimes n}\to (V^*)^{\otimes (m-1)}\otimes V^{\otimes (n-1)}. \] In other words:

\begin{corollary}\label{co.ess_filt}
Let $W\subseteq (V^*)^{\otimes m}\otimes V^{\otimes n}$ be the space of traceless tensors, and set
\begin{equation}\label{eq.filt2}
	W^k=W\cap\left(\sum_{|{\bf r}|\le k}V^{\bf r}\otimes V^{\otimes n}\right),\ \text{ for every } 0\le k\le m.
\end{equation}
The filtration $\{W^k\}$ of $W$ is essential over $\fg^M$. \qedhere
\end{corollary}

We can push this even further, making use of the $S_m\times S_n$-equivariance of the corollary. Recall that the irreducible objects in $\bT_{\End(V)}$ are precisely the modules $W_{\lambda,\mu}$ of traceless tensors in $(V^*)^{\otimes |\lambda|}\otimes V^{\otimes |\mu|}$, for partitions $\lambda,\mu$ (see e.g. \cite[Theorem 4.1]{Mackey} and discussion preceding it). For any pair of partitions, the intersection of \Cref{eq.filt2} (for $m=|\lambda|$ and $n=|\nu|$) with $W_{\lambda,\mu}$ is a filtration of $W_{\lambda,\mu}$ by $\fg^M$-modules. It turns out that it is precisely what we are looking for:

\begin{theorem}\label{th.res2}
For any two partitions $\lambda,\mu$, the intersection of \Cref{eq.filt2} with $W_{\lambda,\mu}$ is the socle filtration of this latter module over $\fg^M$. 
\end{theorem}   
\begin{proof}
Immediate by the proof of \Cref{pr.ess_filt}: simply work with $W_{\lambda,\mu}$ instead of $(V^*)^{\otimes|\lambda|}\otimes V^{\otimes n}$. 
\end{proof}

We now rephrase the theorem slightly, to give a more concrete description of the quotients in the socle filtration. To this end, recall that the ring $\cat{Sym}$ of symmetric functions is a Hopf algebra over $\bZ$ (e.g. $\S$2 of \cite{Haz}), with comultiplication $\Delta$, say. We regard partitions as elements of $\cat{Sym}$ by identifying them with the corresponding Schur functions, and we always think of $\Delta(\lambda)$ as a $\bZ$-linear combination of tensor products $\mu\otimes\nu$ of partitions. 

For a partition $\lambda$ we denote $\Delta(\lambda)$ by $\lambda_{(1)}\otimes\lambda_{(2)}$. Note that this is is a slight notational abuse, as $\Delta(\lambda)$ is not a simple tensor but rather a sum of tensors; we are suppressing the summation symbol to streamline the notation. The summation suppression extends to Schur functors: The expression $M_{\lambda_{(1)}}\otimes M_{\lambda_{(2)}}$, for instance, denotes a direct sum over all summands $\mu\otimes\nu$ of $\Delta(\lambda)$.

Finally, one last piece of notation: For an element $\nu\in\cat{Sym}$ and $k\in\bN$, we denote by $\nu^k \in\cat{Sym}$ the degree-$k$ homogeneous component of $\nu$ with respect to the usual grading of $\cat{Sym}$.  

We can now state the following consequence of \Cref{th.res2}, whose proof we leave to the reader (it consists simply of running through the definition of the comultiplication of $\cat{Sym}$). Recall that we denote simple modules in $\bT_{\End(V)}$ by $W_{\lambda,\mu}$; similarly, simple modules in $\bT_{\fg^M}$ are denoted by $V_{\mu,\nu}$.

\begin{corollary}\label{co.res2}
Let $\lambda,\mu$ be two partitions. The semisimple subquotient $W^k/W^{k-1}$, $k\ge 0$ of the socle filtration $0=W^{-1}\subseteq W^0\subseteq W^1\ldots\ $ of $W_{\lambda,\nu}$ in ${\fg^M}-\mathrm{mod}$ is isomorphic to 
\[
	(V^*/V_*)_{\lambda_{(1)}^k}\otimes V_{\lambda_{(2)}^{|\lambda|-k},\mu}.
\] 
\end{corollary}

Finally, it seems likely that as a category of $\fg^M$-modules, $\bT_{\End(V)}$ has a universal property of its own, reminiscent of \Cref{th.SS}. Denoting by $\bT_{\cat{res}}$ the full (tensor) subcategory of ${\fg^M}-\mathrm{mod}$ on the objects of $\bT_{\End(V)}$, the following seems sensible.

\begin{conjecture}\label{conj.univ}
For any tensor category $\cC$ with monoidal unit ${\bf 1}$, any morphism $b:x\otimes y\to{\bf 1}$ in $\cC$, and any subobject $x'\subseteq x$, there is a left exact tensor functor $F:\bT_{\cat{res}}\to\cC$ sending the pairing $V^*\otimes V\to\bC$ to $b$ and turning the inclusion $V_*\subseteq V^*$ into $x'\subseteq x$. Moreover, $F$ is unique up to tensor natural isomorphism. 
\end{conjecture}



\bibliographystyle{plain}

\end{document}